\documentclass[11pt,a4paper]{article}
\usepackage[T1]{fontenc}
\usepackage{amsfonts,amssymb,amsmath,textcomp,amsthm}
\usepackage{indentfirst}                       
\usepackage{ae}                                
\usepackage{geometry}
\usepackage{graphicx}
\usepackage{hyperref}
\usepackage{float}
\usepackage[dvipsnames]{xcolor}
\usepackage{fancyhdr}
\newtheorem{teo}{Theorem}[section]

\newtheorem{lema}[teo]{Lemma}
\newtheorem{cor}[teo]{Corollary}
\newtheorem{prop}[teo]{Proposition}
\theoremstyle{definition}
\newtheorem{defi}[teo]{Definition}

\newtheorem{exem}[teo]{Example}

\newtheorem{obs}[teo]{Remark}

\def\blfootnote{\xdef\@thefnmark{}\@footnotetext} 

\date{}
\author{GILIARD SOUZA DOS ANJOS\\[0.2cm]Instituto de Matemática e Estatística - Universidade de São Paulo
\\Rua do Matão, 1010, 05508-090, São Paulo - SP, Brazil\\giliards@ime.usp.br}

\title{\Large \textbf{HALF-ISOMORPHISMS WHOSE INVERSES ARE ALSO HALF-ISOMORPHISMS}}
\begin{document}

\maketitle

\begin{abstract} 
\noindent{}Let $(G,*)$ and $(G',\cdot)$ be groupoids. A bijection $f: G \rightarrow G'$ is called a \emph{half-isomorphism} if  $f(x*y)\in\{f(x)\cdot f(y),f(y)\cdot   f(x)\}$, for any $ x, y \in G$. A half-isomorphism of a groupoid onto itself is a half-automorphism. A half-isomorphism $f$ is called \emph{special} if $f^{-1}$ is also a half-isomorphism. In this paper, necessary and sufficient conditions for the existence of special half-isomorphisms on groupoids and quasigroups are obtained. Furthermore, some examples of non-special half-automorphisms for loops of infinite order are provided.

\end{abstract}

\noindent{}{\it Keywords}: half-isomorphism, half-automorphism, special half-isomorphism, groupoid, quasigroup, loop.

\section{Introduction}

A \emph{groupoid} consists of a nonempty set with a binary operation. A groupoid $( Q,*) $ is called a \emph{quasigroup} if for each $ a, b \in Q $ the equations $ a * x  =  b $ and $ y * a =  b$ have unique solutions for $x,y \in Q$. A \textit{quasigroup} $(L,*)$ is a \emph{loop} if there exists an identity element $1\in L$ such that $ 1 * x = x = x *  1$, for any $ x \in L $. The fundamental definitions and facts from groupoids, quasigroups, and loops can be found in \cite{B71,P90}.

Let $(G,*)$ and $(G',\cdot)$ be groupoids. A bijection $f: G \rightarrow G'$ is called a \emph{half-isomorphism} if \mbox{$f(x*y)\in\{f(x)\cdot f(y),f(y)\cdot   f(x)\}$,} for any $ x, y \in G$. A half-isomorphism of a groupoid onto itself is a \emph{half-automorphism}. We say that a half-isomorphism is \emph{trivial} when it is either an isomorphism or an anti-isomorphism.

In 1957, Scott \cite{Sco57} showed that every half-isomorphism on groups is trivial. In the same paper, the author provided an example of a loop of order $8$ that has a nontrivial half-automorphism, then the result for groups can not be generalized to all loops. Recently, a similar version of Scott's result was proved for some subclasses of Moufang loops \cite{GG12,GGRS16,KSV16} and automorphic loops \cite{GA192}. A \emph{Moufang loop} is a loop that satisfies the identity $ x(y(xz))=((xy)x)z$, and an \emph{automorphic loop} is a loop in which every inner mapping is an automorphism \cite{BP56}. We note that there are Moufang loops and automorphic loops that have nontrivial half-automorphisms \cite{GG13,GA19,GPS17}.

In \cite{GA192}, the authors defined the concept of \emph{special half-isomorphism}. A half-isomorphism $f:G \rightarrow G'$ is called \emph{special} if the inverse mapping $f^{-1}:G' \rightarrow G$ is also a half-isomorphism. It is easy to construct an example of a 
half-isomorphism that is not special, as we can see below. 

\newpage
\begin{exem}
\label{ex1} Let $G = \{1,2,...,6\}$ and consider the following Cayley tables of $(G,*)$ and $(G, \cdot)$:

\begin{center}

\begin{minipage}{.35\textwidth}
 \centering
\begin{tabular}{c|cccccc}

$*$ & 1& 2& 3&4&5&6 \\
\hline
 1 &  1&2 &3&4&5&6 \\

 2 & 2&3 &4&5&6&1 \\

 3 &  3&4 &5&6&1&2\\

 4 & 4&5 &6&1&2&3 \\

 5 & 5&6 &1&2&3&4 \\

 6 & 6&1 &2&3&4&5 \\

\end{tabular}

\end{minipage} 
 \begin{minipage}{.35\textwidth}
\centering
\begin{tabular}{c|cccccc}

$\cdot $ & 1& 2& 3&4&5&6 \\
\hline
 1 &  1&2 &3&4&5&6 \\

 2 & 2&3 &4&5&6&1 \\

 3 &  3&1 &5&6&4&2\\

 4 & 4&5 &6&1&2&3 \\

 5 & 5&6 &1&2&3&4 \\

 6 & 6&4 &2&3&1&5 \\

\end{tabular}

\end{minipage}
\end{center}

Note that $(G,*)$ is isomorphic to $C_6$, the cyclic group of order $6$, and $(G,\cdot)= L$ is a nonassociative loop. Consider the mapping $f: C_6\rightarrow L$ defined by $f(x) = x$, for all $x\in G$. For $x,y\in G$ such that $x\leq y$ and $(x,y)\not = (3,5)$, we have $y*x = x*y = x\cdot y$. Furthermore, $3*5 = 5*3 = 5\cdot 3$. Thus, $f$ is a half-isomorphism. From $3\cdot 5 = 4$ and $3*5 = 5*3 = 1$, it follows that $f^{-1}(3\cdot 5)\not \in \{f^{-1}(3)*f^{-1}(5),f^{-1}(5)*f^{-1}(3)\}$, and hence $f^{-1}$ is not a half-isomorphism.
\qed 
\end{exem}

We note that providing some examples for the case of non-special half-automorphisms can be very complicated. For finite loops, every half-automorphism is special \cite[Corollary 2.7]{GA192}, and in section~\ref{sec3} we show that the same is valid for finite groupoids. 

As we can see in the example~\ref{ex1}, in general, a half-isomorphism does not preserve the structure of the loop. For instance, $C_6$ is associative and commutative and has a subgroup $H = \{1,3,5\}$, while $L$ is nonassociative and noncommutative, and $f(H)$ is not a subloop of $L$. However, the inverse mapping of a half-isomorphism can preserve some structure, like the commutative property and subloops \cite[Proposition 2.2]{GA192}. The same naturally holds for special half-isomorphisms.

This paper is organized as follows: Section~\ref{sec2} presents the definitions and basic results about half-isomorphisms. In section~\ref{sec3}, some presented results in \cite{GA192} on half-isomorphisms in loops are generalized to groupoids. In section~\ref{sec4}, the concept of \emph{principal h-groupoid} of a groupoid is defined, and then a necessary and sufficient condition for the existence of special half-isomorphisms between groupoids is obtained. Furthermore, equations related to the number of special half-automorphisms, automorphisms and anti-automorphisms of a groupoid are obtained. In section~\ref{sec5}, the concept of \emph{principal h-quasigroup} of a quasigroup is defined, and then the set of these quasigroups is described. Some examples of non-special half-automorphisms in loops are provided in section~\ref{sl}.

\section{Preliminaries}
\label{sec2}

Here, the required definitions and basic results on half-isomorphisms are stated.

\begin{defi} Let $G$ and $G'$ be groupoids. We will say that $G$ is \emph{half-isomorphic} to $G'$, denoted by $G\stackrel{H}{\cong} G'$, if there exists a special half-isomorphism between $G$ and $G'$. Note that $\stackrel{H}{\cong}$ is an equivalence relation. If $G$ is isomorphic to $G'$, we write $G \cong G'$.
\end{defi}

The next proposition assures that quasigroups half-isomorphic to loops are also loops.

\begin{prop}
\label{prop21}
Let $(G,*)$ and $(G',\cdot)$ be groupoids and $f:G \rightarrow G'$ be a half-isomorphism. If $G'$ has an identity element $1$, then $f^{-1}(1)$ is the identity element of $G$.
\end{prop}
\begin{proof}
Let $x = f^{-1}(1) \in G$. For $y\in G$, we have that $\{f(x*y),f(y*x)\}\subset \{1\cdot f(y),f(y)\cdot1\} = \{f(y)\}$. Since $f$ is a bijection, we have $x*y = y*x = y$. Therefore, $x$ is an identity element of $G$.
\end{proof}

Now, let $(G,*),(G',\cdot),(G'',\bullet)$ be groupoids, and $f:G \rightarrow G'$ and $g:G' \rightarrow G''$ be half-isomorphisms. For $x,y\in G$, we have

\begin{center}
$gf(x*y) \in \{g(f(x)\cdot f(y)),g(f(y)\cdot f(x))\} = \{gf(x)\bullet gf(y),gf(y)\bullet gf(x))\}.$
\end{center}

Thus, $gf$ is a half-isomorphism. If $f$ and $g$ are special half-isomorphisms, then $(gf)^{-1} = f^{-1}g^{-1}$ is also a special half-isomorphism.

We denote the sets of the half-automorphisms, special half-automorphisms, and trivial half-automorphisms of a groupoid $G$ by $\mathit{Half}(G)$, $\mathit{Half}_S(G)$, and $\mathit{Half_T}(G)$, respectively. Note that automorphisms and anti-automorphisms are always special half-automorphisms, and consequently $\mathit{Half_T}(G) \subset \mathit{Half}_S(G) \subset \mathit{Half}(G)$. 

For $f,g\in \mathit{Half}(G)$, we already see that $fg\in \mathit{Half}(G)$. The identity mapping $I_d$ of $G$ is the identity element of $\mathit{Half}(G)$. Thus, $\mathit{Half}(G)$ is a group if and only if it is closed under inverses, which is equivalent to $\mathit{Half}(G) = \mathit{Half}_S(G)$. In particular, $\mathit{Half}_S(G)$ is always a group.

A composition of two automorphisms or two anti-automorphisms is an automorphism, and if $f$ is an automorphism and $g$ is an anti-automorphism, then $fg$ and $gf$ are anti-automorphisms and $g^{-1}fg$ is an automorphism. Thus, $\mathit{Half_T}(G)$ is a group and the automorphism group of $G$, denoted by $Aut(G)$, is a normal subgroup of $\mathit{Half_T}(G)$. 

The following result summarizes the discussion above.

\begin{prop}
\label{prop22}
Let $G$ be a groupoid. Then:

\noindent{}(a) $\mathit{Half}_S(G)$ is a group and $\mathit{Half}_T(G)$ is a subgroup of $\mathit{Half}_S(G)$.
\\
(b) $\mathit{Half}(G)$ is a group if and only if $\mathit{Half}(G)=\mathit{Half}_S(G)$.
\\
(c) $Aut(G) \triangleleft \mathit{Half}_T(G)$.
\end{prop}

\begin{obs}
It is shown in section~\ref{sl} that in general $\mathit{Half}(G)$ is not a group. 
\end{obs}

\section{Special half-isomorphisms on groupoids}
\label{sec3}

Considering $(G,*)$ and $(G',\cdot)$ as groupoids, define the following set:

\begin{equation*}
K(G) = \{(x,y)\in G\times G \mid xy = yx\}
\end{equation*}

The next two results are respectively extensions of Proposition 2.3 and Theorem 2.5 of \cite{GA192} to groupoids. We note that the proofs are similar to the ones for corresponding results given in \cite{GA192}.

\begin{lema}
\label{lema32}
Let $f:G \rightarrow G'$ be a half-isomorphism. Then 

\begin{equation*}
\begin{aligned}
\psi_{(G,G')}:{} & K(G') &\rightarrow {}&K(G)\\
&(x,y)&\mapsto  {}&(f^{-1}(x),f^{-1}(y))
\end{aligned}
\end{equation*}

\noindent{}is injective.
\end{lema}
\begin{proof}

For $(x,y)\in K(G')$, we have
\begin{equation*}
\{f(f^{-1}(x)*f^{-1}(y)),f(f^{-1}(y)*f^{-1}(x))\} \subseteq \{x\cdot y, y\cdot x\} = \{x\cdot y\}.
\end{equation*}
Then, $f(f^{-1}(x)*f^{-1}(y))=f(f^{-1}(y)*f^{-1}(x))$, and so $f^{-1}(x)*f^{-1}(y) = f^{-1}(y)*f^{-1}(x)$. Thus, $(f^{-1}(x),f^{-1}(y))\in K(Q)$ and the mapping $\psi_{(G,G')}$ is well-defined. 

Now, let $(x,y),(x',y')\in K(G')$ such that $\psi_{(G,G')}((x,y)) = \psi_{(G,G')}((x',y'))$. Then, $f^{-1}(x) = f^{-1}(x')$ and $f^{-1}(y) = f^{-1}(y')$. Since $f$ is a bijection, the mapping $\psi_{(G,G')}$ is injective.
\end{proof}

\begin{teo}
\label{teo31}
Let $f:G \rightarrow G'$ be a half-isomorphism. Then, the following statements are equivalent:

\noindent{}(a) $f$ is special.\\
(b) $\{f(x*y),f(y*x)\} = \{f(x)\cdot f(y), f(y)\cdot f(x)\}$ for any $x,y\in G$.\\
(c) For all $x,y\in G$ such that $x*y = y*x$, we have $f(x)\cdot f(y) = f(y)\cdot f(x)$.\\
(d) $\psi_{(G,G')}$ is a bijection.

\end{teo}

\begin{proof}(a) $\Rightarrow$ (b) Let $x,y \in G$. Since $f$ is a half-isomorphism, we have \mbox{$\{f(x*y),f(y*x)\} \subseteq$} \mbox{$ \{f(x)\cdot f(y), f(y)\cdot f(x)\}$.} Since $f^{-1}$ is a half-isomorphism, we have $\{f^{-1}(f(x)\cdot f(y)), f^{-1}(f(y)\cdot f(x))\} \subseteq \{x*y,y*x\}$, and hence $\{f(x)\cdot f(y), f(y)\cdot f(x)\} \subseteq \{f(x*y),f(y*x)\}$.\\

\noindent{}(b) $\Rightarrow$ (c) Let $x,y \in G$ such that $x*y = y*x$. Then, $f(x*y) = f(y*x)$. Using the hypothesis, we get $\{f(x)\cdot f(y), f(y)\cdot f(x)\} = \{f(x*y),f(y*x)\} = \{f(x*y)\}$, and therefore $f(x)\cdot f(y) = f(y)\cdot f(x)$.\\

\noindent{}(c) $\Rightarrow$ (d) From Lemma~\ref{lema32}, we know that $\psi_{(G,G')}$ is  injective. Let $(x,y) \in K(G)$. By hypothesis, we have $f(x)\cdot f(y) = f(y)\cdot f(x)$, and then $(f(x),f(y)) \in K(G')$. It is clear that $\psi_{(G,G')}((f(x),f(y))) = (x,y)$, and hence $\psi_{(G,G')}$ is a bijection.\\

\noindent{}(d) $\Rightarrow$ (a) Let $x,y\in G'$. If $(x,y)\in K(G')$, then $(f^{-1}(x),f^{-1}(y)) \in K(G)$ since $\psi_{(G,G')}$ is a bijection. Thus, $f(f^{-1}(x)* f^{-1}(y)) = x\cdot y$, and therefore $f^{-1}(x\cdot y) = f^{-1}(x)* f^{-1}(y)$. If $(x,y)\not \in K(G')$, then  $(f^{-1}(x),f^{-1}(y)) \not \in K(G)$ since $\psi_{(G,G')}$ is a bijection. Consequently, we have 

\begin{center}
$\{f(f^{-1}(x)* f^{-1}(y)), f(f^{-1}(y)* f^{-1}(x))\} = \{x\cdot y,y\cdot x\}$,
\end{center}

\noindent{}and hence $f^{-1}(x\cdot y) \in \{f^{-1}(x)* f^{-1}(y), f^{-1}(y)* f^{-1}(x)\}$.
\end{proof}

As direct consequences of Lemma~\ref{lema32} and Theorem~\ref{teo31}, we have the following corollaries.

\begin{cor}
\label{cor31}
Let $f:G \rightarrow G'$ be a half-isomorphism. If $|K(G)| = |K(G')| < \infty$, then $f$ is special.
\end{cor}

\begin{cor}
\label{cor32}
Let $G$ be a groupoid such that $|K(G)| < \infty$. Then, $\mathit{Half}(G)$ is a group.
\end{cor}

\begin{cor}
\label{cor33}
Let $G$ be a finite groupoid. Then, $\mathit{Half}(G)$ is a group.
\end{cor}

A loop is \emph{diassociative} if any two of its elements generate an associative subloop. Moufang loops and groups are examples of diassociative loops. In \cite[Lemma 2.1]{KSV16}, the authors showed that the item (c) of  Theorem~\ref{teo31} holds for any half-isomorphism on diassociative loops. Therefore, we have the next result.

\begin{cor}
\label{cor34} 
Let $(L,*)$ and $(L',\cdot)$ be diassociative loops. Then, every half-isomorphism between $L$ and $L'$ is special.
\end{cor}

\begin{obs} The Corollary~\ref{cor34} cannot be extended for some important classes of loops. In example~\ref{ex61}, a non-special half-isomorphism between a right Bol loop and a group is introduced. A loop is called \emph{right Bol loop} if it satisfies the identity $((xy)z)y = x((yz)y)$.
\end{obs}

This section is finished with a property of half-isomorphic groupoids.

\begin{prop}
\label{prop33} 
If $G\stackrel{H}{\cong} G'$, then: 

\noindent{}(a) $\mathit{Half}(G) \cong \mathit{Half}(G')$
\\
(b) $\mathit{Half}_S(G) \cong \mathit{Half}_S(G')$

\end{prop}
\begin{proof}
Let $\phi: G \rightarrow G'$ be a special half-isomorphism. Define $\varphi: \mathit{Half}(G) \rightarrow \mathit{Half}(G')$ by $\varphi(f) = \phi f \phi^{-1}$. It is clear that $\varphi$ is a bijection. For $f,g\in \mathit{Half}(G)$, we have $\varphi(fg) = \phi fg \phi^{-1} = \phi f\phi^{-1} \phi g \phi^{-1} = \varphi(f)\varphi(g)$. Thus, $\mathit{Half}(G) \cong \mathit{Half}(G')$. The rest of the claim is concluded from the fact that $\varphi(\mathit{Half}_S(G)) = \mathit{Half}_S(G')$.
\end{proof}

\begin{obs}
\label{ob31} If $G\stackrel{H}{\cong} G'$, then $Aut(G)$ is not isomorphic to $Aut(G')$ in general (see example~\ref{ex41}).
\end{obs}

\section{Principal h-groupoids of G}
\label{sec4}

In this section, $G_0 = (G,*)$ is considered as a noncommutative groupoid.

Let $(G',\bullet)$ be a groupoid such that $G_0\stackrel{H}{\cong} (G',\bullet)$. Then, there exists a special half-isomorphism $f$ of $G_0$ into $(G',\bullet)$. Define an operation $\cdot$ on $G$ by $x\cdot y = f^{-1}(f(x)\bullet f(y))$. Thus, $f$ is an isomorphism of $(G,\cdot)$ into $(G',\bullet)$, and hence $I_d: G_0 \rightarrow (G,\cdot)$ is a special half-isomorphism, where $I_d$ is the identity mapping of $G$.

A groupoid $(G,\cdot)$ for which $I_d: G_0 \rightarrow (G,\cdot)$ is a special half-isomorphism is called a \emph{principal h-groupoid of $G_0$}. Therefore, the following result is at hand.

\begin{prop}
\label{prop41}
Let $G'$ be a groupoid. Then, $G_0\stackrel{H}{\cong} G'$ if and only if $G'$ is isomorphic to a principal h-groupoid of $G_0$.
\end{prop}

Denote by $\mathcal{M}(G_0)$ the set of the principal h-groupoids of $G_0$. Note that for \mbox{$(G,\cdot),(G,\bullet)\in \mathcal{M}(G_0)$,} we have $(G,\cdot)=(G,\bullet)$ if $x\cdot y = x \bullet y$, for all $x,y\in G$, which is equivalent to $I_d$ being an isomorphism between $(G,\cdot)$ and $(G,\bullet)$.

Let $(G,\cdot)\in \mathcal{M}(G_0)$. Since $I_d: G_0 \rightarrow (G,\cdot)$ is a special half-isomorphism, we have

\begin{equation}
\label{eq41}
\{x*y,y*x\} = \{x\cdot y,y\cdot x\}, \textrm{ for all } x,y \in G.
\end{equation}

If $(x,y)\in K(G_0)$, then $x\cdot y = y\cdot x = x*y$. For each pair $(x,y),(y,x)\in G\times G\setminus K(G_0)$, there are two possible values for $x\cdot y$ and $y\cdot x$ by \eqref{eq41}. Thus, if $G$ is finite, we have $2^{| G\times G\setminus K(G_0)|/2}$ possibilities for a principal h-groupoid of $G_0$. Hence, the following result is at hand.

\begin{prop}
\label{prop42} If $G$ is finite, then $|\mathcal{M}(G_0)| =2^{(|G|^2-|K(G_0)|)/2}$.

\end{prop}

Define $\mathcal{M}_I(G_0) = \{G'\in \mathcal{M}(G_0)\,|\, G'\cong G_0\}$ and let $S(G)$ be the set of permutations of $G$. For $G' = (G,\cdot) \in \mathcal{M}_I(G)$, define $Iso(G',G_0) = \{f\in S(G)\,|\, f \textrm{ is an isomorphism of } G' \textrm{ into } G_0\}$. Note that $Iso(G_0,G_0) = Aut(G_0)$. In the next result, we determine a relationship between $\mathit{Half}_S(G_0)$, $Aut(G_0)$ and $\mathcal{M}_I(G_0)$.

\begin{prop}
\label{prop43} We have:

\noindent{}(a) $Iso(G',G_0) \subset \mathit{Half}_S(G_0)$, for every $G'\in \mathcal{M}_I(G_0)$.
\\
(b) For each $G'\in \mathcal{M}_I(G_0)$, $Iso(G',G_0)$ is a right coset of $Aut(G_0)$ in $\mathit{Half}_S(G_0)$, that is, there exists $f\in \mathit{Half}_S(G_0)$ such that $Iso(G',G_0) = Aut(G_0)f$.
\\
(c) For $G_1,G_2 \in \mathcal{M}_I(G_0)$, if $Iso(G_1,G_0) \cap Iso(G_2,G_0) \not = \emptyset$, then $G_1 = G_2$.
\\
(d) $\mathit{Half}_S(G_0) = \bigcup_{G'\in \mathcal{M}_I(G_0)} Iso(G',G_0)$.
\\
(e) $|\mathcal{M}_I(G_0)| = [\mathit{Half}_S(G_0):Aut(G_0)]$, which is the index of $Aut(G_0)$ in $\mathit{Half}_S(G_0)$.

\end{prop}
\begin{proof} (a) For $G' = (G,\cdot)\in \mathcal{M}_I(G_0)$, let $f\in Iso(G',G_0)$. Then $f(x\cdot y) = f(x)*f(y)$, for all $x,y\in G$. By \eqref{eq41}, $\{f(x\cdot y),f(y\cdot x)\} = \{f(x)*f(y),f(y)*f(x)\}$, for all $x,y \in G$. By Theorem~\ref{teo31}, $f\in \mathit{Half}_S(G_0)$.\\

\noindent{}(b) Fix $f\in Iso(G',G_0)$. It is clear that $gf^{-1}\in Aut(G_0)$, for every $g\in Iso(G',G_0)$, and $\alpha f \in Iso(G',G_0)$, for every $\alpha \in Aut(G_0)$. Hence, we have the desired result.\\

\noindent{}(c) Let $f\in Iso(G_1,G_0) \cap Iso(G_2,G_0)$. Note that $I_d= f^{-1}f:G_1 \rightarrow G_2$ is an isomorphism. From the definition of $\mathcal{M}(G_0)$, it follows that $G_1 = G_2$.\\

\noindent{}(d) Let $f \in \mathit{Half}_S(G_0)$. Define the operation $\cdot$ on $G$ by $x\cdot y = f^{-1}(f(x)*f(y))$, for all $x,y\in G$. Note that $f: (G,\cdot)\rightarrow (G,*)$ is an isomorphism. Furthermore, since $f \in \mathit{Half}_S(G_0)$, and $f(x\cdot y) = f(x)*f(y)$ and $f(y\cdot x) = f(y)*f(x)$, for all $x,y \in G$, we have $\{x\cdot y,y\cdot x\} = \{x*y,y*x\}$, for all $x,y \in G$. Thus, $G' = (G,\cdot)\in \mathcal{M}_I(G_0)$, and hence $f\in Iso(G',G_0)$.\\ 

\noindent{}(e) It is a consequence of the previous items.
\end{proof}

As a consequence of the Proposition~\ref{prop33} and the item (e) of Proposition~\ref{prop43}, we have the following result.

\begin{cor}
\label{cor41} Let $G',G''$ be groupoids such that $G'\stackrel{H}{\cong} G''$ and $\mathit{Half}_S(G')$ is finite. Then,

\begin{center}
$|\mathcal{M}_I(G')|.|Aut(G')| = |\mathcal{M}_I(G'')|.|Aut(G'')|$
\end{center}

\end{cor}

Define $G_0^T = (G,\cdot)$, where $x\cdot y = y*x$, for all $x,y\in G$, and denote the set of anti-automorphisms of $G_0$ by $Ant(G_0)$. Since $G_0$ is noncommutative, we have $Aut(G_0)\cap Ant(G_0) = \emptyset$.

\begin{prop}
\label{prop44} $G_0$ has an anti-automorphism if and only if $G_0^T\in \mathcal{M}_I(G_0)$. In this case, $|Ant(G)| = |Aut(G)|$.
\end{prop}
\begin{proof} Note that a bijection $f$ of $G$ is an anti-automorphism of $G_0$ if and only if $f$ is an isomorphism of $G_0$ into $G_0^T$. The rest of the claim is concluded from the item (b) of Proposition~\ref{prop43}.
\end{proof}

\begin{exem} 
\label{ex41}
Let $Q = \{1,2,...,8\}$ and consider the following Cayley tables of $(Q,*)$ and $(Q,\cdot)$:

\begin{center}

\begin{minipage}{.4\textwidth}
\centering
\begin{tabular}{c|cccccccc}

$*$&1&2&3&4&5&6&7&8\\
\hline
1&1&2&3&4&6&5&7&8\\

2&2&1&4&3&5&6&8&7\\

3&4&3&1&2&7&8&5&6\\

4&3&4&2&1&8&7&6&5\\

5&5&6&8&7&1&2&4&3\\

6&6&5&7&8&2&1&3&4\\

7&8&7&6&5&3&4&1&2\\

8&7&8&5&6&4&3&2&1\\

\end{tabular}

\end{minipage} 
 \begin{minipage}{.4\textwidth}
\centering
\begin{tabular}{c|cccccccc}
$\cdot $&1&2&3&4&5&6&7&8\\
\hline
1&1&2&4&3&6&5&7&8\\
2&2&1&3&4&5&6&8&7\\
3&3&4&1&2&7&8&5&6\\
4&4&3&2&1&8&7&6&5\\
5&5&6&8&7&1&2&4&3\\
6&6&5&7&8&2&1&3&4\\
7&8&7&6&5&3&4&1&2\\
8&7&8&5&6&4&3&2&1\\
\end{tabular}

\end{minipage}

\end{center}

We have $(Q,*)$ and $(Q,\cdot)$ being quasigroups. Note that, for $x,y\in Q$:

\begin{center}
$x*y = \left\{\begin{array}{l}
y\cdot x,  \textrm{ if } (x,y) \in \{(1,3),(1,4),(2,3),(2,4),(3,1),(3,2),(4,1),(4,2)\},\\
x\cdot y,  \textrm{ otherwise.} 
\end{array}\right.$
\end{center}

Thus, $(Q,\cdot) \in \mathcal{M}((Q,*))$. Using the LOOPS package \cite{NV1} for GAP \cite{gap} we get $|Aut((Q,*))| = 4$ and $|Aut((Q,\cdot))| = 8$. This illustrates Remark~\ref{ob31}.

Note that $|K((Q,*))| = 16$, and hence $|\mathcal{M}((Q,*))| = 2^{24} = 16777216$. Using a GAP computation with the LOOPS package, we get that there are $64$ quasigroups in $\mathcal{M}((Q,*))$ and $|\mathcal{M}_I((Q,*))| = 12$. By Proposition~\ref{prop43}, we have $|\mathit{Half}((Q,*))|= 48$ and $|\mathcal{M}_I((Q,\cdot))| = 6$.

It is observed that the number of quasigroups in $\mathcal{M}((Q,*))$ is much smaller than $|\mathcal{M}((Q,*))|$. In the next section, we will see that the same occurs for any finite noncommutative quasigroup. \qed
\end{exem}

\section{Principal h-quasigroups of Q}
\label{sec5}

Here, $Q_0= (Q,*)$ is considered as a noncommutative quasigroup. A quasigroup $(Q,\cdot)$ is a \emph{principal h-quasigroup of $Q_0$} if $(Q,\cdot)\in \mathcal{M}(Q_0)$. Denote by $\mathcal{N}(Q_0)$ the set of the principal h-quasigroups of $Q_0$. It is clear that $\mathcal{M}_I(Q_0) \subset \mathcal{N}(Q_0) \subset \mathcal{M}(Q_0)$. The next result is concluded from Proposition~\ref{prop41}.

\begin{prop}
\label{prop51} Let $Q'$ be a quasigroup. Then $Q_0\stackrel{H}{\cong} Q'$ if and only if $Q'$ is isomorphic to a principal h-quasigroup of $Q_0$.
\end{prop}

Now, we describe $\mathcal{N}(Q_0)$. For $(x,y),(x',y')\in Q\times Q \setminus K(Q_0)$, we say that $(x,y)\sim (x',y')$ if one of the following holds:

\noindent{}(i) $(x',y') = (y,x)$,
\\
(ii) $x = x'$ and $\{x*y,y*x\}\cap \{x*y',y'*x\}\not = \emptyset$,
\\
(iii) $y = y'$ and $\{x*y,y*x\}\cap \{x'*y,y*x'\}\not = \emptyset$.

We say that $(x,y)\equiv (x',y')$ if there are $z_1,z_2,...,z_l \in Q\times Q \setminus K(Q_0)$ such that $(x,y)\sim z_1 \sim z_2\sim...\sim z_l \sim  (x',y')$.

The relation $\sim$ is reflexive and symmetric, and hence $\equiv$ is an equivalence relation. Denote by $r(Q_0)$ the number of equivalence classes of $\equiv$ on $Q\times Q \setminus K(Q_0)$. 

Suppose that $Q$ is finite and let $\tau = \{(x_1,y_1),(x_2,y_2),...,(x_{r(Q_0)},y_{r(Q_0)})\}$ be a set of representatives of the equivalence classes of $\equiv$ on $Q\times Q \setminus K(Q_0)$. Consider $\mathbb{Z}_2 = \{0,1\}$, and for $\sigma = \{\sigma_1,\sigma_2,...,\sigma_{r(Q_0)}\} \in \mathbb{Z}_2^{r(Q_0)}$, define the operation $\stackrel{\sigma}{\bullet}$ on $Q$ by:

\begin{center}
$x \stackrel{\sigma}{\bullet} y = \left\{\begin{array}{l}
x * y,  \textrm{ if } (x,y)\in K(Q_0) \textrm{ or } (x,y)\equiv (x_i,y_i), \textrm{ where } \sigma_i = 0,\\
y * x,  \textrm{ if } (x,y)\equiv (x_i,y_i), \textrm{ where } \sigma_i = 1.
\end{array}\right.$
\end{center}

Denote $(Q,\stackrel{\sigma}{\bullet})$ by $Q_\sigma$ and let $\mathcal{N}_\tau(Q_0) = \{Q_\sigma \,|\, \sigma  \in \mathbb{Z}_2^{r(Q_0)}\}$. Note that $\mathcal{N}_\tau(Q_0)\subset \mathcal{M}(Q_0)$ and $|\mathcal{N}_\tau(Q_0)| = 2^{r(Q_0)}$.

\begin{teo}
\label{teo51} If $Q$ is finite, then $\mathcal{N}(Q_0) = \mathcal{N}_\tau(Q_0)$. In particular, $|\mathcal{N}(Q_0)| = 2^{r(Q_0)}$.

\end{teo}
\begin{proof}

Let $Q_\sigma \in \mathcal{N}_\tau(Q_0)$. Since $Q$ is finite, in order to prove that $Q_\sigma$ is a quasigroup, we only need to show that the cancellation laws are satisfied, that is, $x \stackrel{\sigma}{\bullet} y = x \stackrel{\sigma}{\bullet} y' \Rightarrow y = y'$ and $x \stackrel{\sigma}{\bullet} y = x' \stackrel{\sigma}{\bullet} y \Rightarrow x = x'$. 

Let $x,y,y'\in Q$ be such that $x \stackrel{\sigma}{\bullet} y = x \stackrel{\sigma}{\bullet} y'$. If $(x,y)\in K(Q_0)$, then \mbox{$x*y = y*x \in $} $ \{x*y',y'*x\}$, and hence $y = y'$. Now suppose that $(x,y)\not \in K(Q_0)$. We have four possibilities:

\noindent{}(i) $x \stackrel{\sigma}{\bullet} y = x*y$ and $x \stackrel{\sigma}{\bullet} y' = x*y'$,
\\
(ii) $x \stackrel{\sigma}{\bullet} y = y*x$ and $x \stackrel{\sigma}{\bullet} y' = y'*x$,
\\
(iii) $x \stackrel{\sigma}{\bullet} y = x*y$ and $x \stackrel{\sigma}{\bullet} y' = y'*x$,
\\
(iv) $x \stackrel{\sigma}{\bullet} y = y*x$ and $x \stackrel{\sigma}{\bullet} y' = x*y'$.

In (i) and (ii), it is immediately seen that $y = y'$.

For (iii) and (iv), we have  $(x,y)\sim (x,y')$. Hence, there exists $(x_i,y_i)\in \tau$ such that $(x,y) \equiv (x_i,y_i)$ and $(x,y') \equiv (x_i,y_i)$. By definition of $\stackrel{\sigma}{\bullet}$, we have either $x \stackrel{\sigma}{\bullet} y = x*y$ and $x \stackrel{\sigma}{\bullet} y' = x*y'$, or $x \stackrel{\sigma}{\bullet} y = y*x$ and $x \stackrel{\sigma}{\bullet} y' = y'*x$. Since $(x,y)\not \in K(Q_0)$, it follows that $(x,y')\in K(Q_0)$. Similarly to the case $(x,y)\in K(Q_0)$, one can conclude that $y = y'$.

Thus, the cancellation law $x \stackrel{\sigma}{\bullet} y = x \stackrel{\sigma}{\bullet} y' \Rightarrow y = y'$ holds in $Q_\sigma$. The second cancellation law can be proven similarly. Therefore, $Q_\sigma \in \mathcal{N}(Q_0)$.

Conversely, let $Q' = (Q,\cdot) \in \mathcal{N}(Q_0)$. Then, there exists $\sigma \in \mathbb{Z}_2^{r(Q_0)}$ such that $x_i\cdot y_i = x_i \stackrel{\sigma}{\bullet} y_i$, for any $(x_i,y_i)\in \tau$. For $(x,y)\in K(Q_0)$, it is vividly deduced that $x\cdot y = x \stackrel{\sigma}{\bullet} y$. 

Consider $(x_i,y_i)\in \tau$. Then, $y_i\cdot x_i = y_i \stackrel{\sigma}{\bullet} x_i$. Let $(x,y) \in Q\times Q\setminus \{(x_i,y_i),(y_i,x_i)\}$ such that $(x,y)\sim (x_i,y_i)$. By \eqref{eq41} and the definition of $\stackrel{\sigma}{\bullet}$, we have $x\cdot y \not = x_i\cdot y_i = x_i\stackrel{\sigma}{\bullet}  y_i$ and $x \stackrel{\sigma}{\bullet} y \not = x_i\stackrel{\sigma}{\bullet}  y_i$, and therefore the only possibility is $x\cdot y  = x \stackrel{\sigma}{\bullet} y$. For every $(x,y)\sim (x_i,y_i)$, one can use the previous arguments and result in $x'\cdot y'  = x' \stackrel{\sigma}{\bullet} y'$, for all $(x',y')\sim (x,y)$. Since $Q$ is finite, this procedure must end at some point, and hence  $x\cdot y  = x \stackrel{\sigma}{\bullet} y$, for all $(x,y)\equiv (x_i,y_i)$. As a result, we have $Q' = Q_\sigma$.
\end{proof}

By Proposition~\ref{prop42}, if $Q$ is finite, then $r(Q_0)\leq (|Q|^2 - |K(Q_0)|)/2$. The next proposition provides a better estimate for $r(Q_0)$. According to this result, it is seen that $|\mathcal{N}(Q_0)|$ is much smaller that $|\mathcal{M}(Q_0)|$.

\begin{prop}
\label{prop52a} If $Q$ is finite, then $r(Q_0)\leq (|Q|^2 - |K(Q_0)|)/6$ and $|\mathcal{N}(Q_0)|\leq  \sqrt[3]{|\mathcal{M}(Q_0)|}$. In particular, $|\mathcal{M}(Q_0)| \geq 8$.
\end{prop}
\begin{proof}
Let $(x,y)\in Q\times Q \setminus K(Q_0)$ and $[(x,y)]$ be the equivalence class of $(x,y)$ with respect to $\equiv$. Since $Q_0$ is a quasigroup, there are $x',y' \in Q$ such that $x'\not =x$, $y'\not = y$, $(x',y)\sim (x,y)$, and $(x,y')\sim (x,y)$. We have $x\not = y$, $x'\not = y$ and $x\not = y'$ since $(x,y),(x',y),(x,y')\not \in K(Q_0)$. Thus, $|[(x,y)]|\geq |\{(x,y),(x',y),(x,y'),(y,x),(y,x'),(y',x)\}| = 6$. Hence, $|Q\times Q \setminus K(Q_0)|\geq 6 \,r(Q_0)$. The rest of the claim follows from Proposition~\ref{prop42}, Theorem~\ref{teo51} and the fact that $r(Q_0)\geq 1$.
\end{proof}

If $Q$ is finite and $r(Q_0)$ is small, one can generate all quasigroups of $\mathcal{N}(Q_0)$ computationally. Then, by using Propositions~\ref{prop51} and \ref{prop44} it can be verified if a quasigroup $Q'$ is half-isomorphic to $Q_0$ and generated all elements of $\mathit{Half}(Q_0)$. However, $r(Q_0)$ can be a large number even for groups of small order, and therefore generating all the quasigroups of $\mathcal{N}(Q_0)$ becomes computationally unviable. The next example illustrates both situations. In this example, $r(Q_0)$ and $|\mathcal{M}(Q_0)|$ are obtained by using GAP computing with the LOOPS package \cite{gap,NV1}.

\begin{exem} (a) Let $A_5$ be the alternating group of order $60$. We have that $r(A_5) = 91$, and hence $|\mathcal{N}(A_5)| = 2^{91}$. Furthermore, $|\mathcal{M}(A_5)| = 2^{1650}$. 

\noindent{}(b) The LOOPS package for GAP contains all nonassociative right Bol loops of order $141$ (there are $23$ such loops). The right Bol loops of this order were classified in \cite{KNV}. If $L$ is one of these loops, then $3\leq r(L) \leq 8$, and hence $|\mathcal{N}(L)| \leq 256$. Furthermore, $|\mathcal{M}(L)| \geq 2^{5405}$.\qed

\end{exem}

By Proposition~\ref{prop21}, every quasigroup half-isomorphic to a loop is also a loop. Consequently, the same results as those presented for quasigroups in this section can be proven for loops. For more structured classes of loops, as it is seen in the following result, one can provide more information about the loops of $\mathcal{N}(L)$.

\begin{prop}
\label{prop52}
Let $G$ be a finite noncommutative group. Then, $|\mathcal{M}_I(G)| = 2$.
\end{prop}
\begin{proof}
From Scott's result \cite{Sco57}, we have $\mathit{Half}(G) = \mathit{Half}_T(G)$. Since $G$ is noncommutative, the mapping $J: G \to G$, defined by $J(x) = x^{-1}$, is an anti-automorphism of $G$. By Proposition~\ref{prop44}, we have $|\mathit{Half}(G)| = 2|Aut(G)|$. Thus, the claim follows from Proposition~\ref{prop43}.
\end{proof}

In fact, the previous proposition can be extended to any noncommutative loop that has an anti-automorphism and where every half-automorphism is trivial, such as the noncommutative loops of the subclass of Moufang loops in \cite[Thereom 1.4]{KSV16}, which include the noncommutative Moufang loops of odd order \cite{GG12}. Notice that this result cannot be extended even to all Moufang loops. In \cite[Example 4.6]{G20}, a noncommutative Moufang loop $L$ of order $16$ is given for which $|\mathcal{M}_I(L)| = [\mathit{Half}(L):Aut(L)] = 16.$

\section{A construction of a non-special half-automorphism}
\label{sl}
Let $G$ be a nonempty set with binary operations $*$ and $\cdot$ such that there exists a non-special half-isomorphism $f:(G,*) \rightarrow (G,\cdot)$. Define $G_\infty = \prod_{i=1}^\infty G$. The elements of $G_\infty$ will be denoted by $(x_i) = (x_i)_{i=1}^\infty$, where $x_i \in G$, for all $i$. For $(x_i),(y_i)\in G_\infty$, define the operation $(x_i)\bullet (y_i) = (z_i)$, where

\begin{center}
$z_j = \left\{\begin{array}{l}
x_j*y_j,  \textrm{ if } j \textrm{ is odd}, \\
x_j\cdot y_j,  \textrm{ if } j \textrm{ is even.} \\
\end{array}\right.$
\end{center}

Then, $(G_\infty,\bullet)$ is a groupoid. It is easy to see that if $(G,*)$ and $(G,\cdot)$ are quasigroups (loops), then $(G_\infty,\bullet)$ is also a quasigroup (loop). Define the mapping $\phi: G_\infty\rightarrow G_\infty$ by $\phi(x_i) = (y_i)$, where

\begin{center}
$y_j = \left\{\begin{array}{l}
f(x_1),  \textrm{ if } j = 2, \\
x_{j+2},  \textrm{ if } j \textrm{ is odd}, \\
x_{j-2},  \textrm{ if } j>2 \textrm{ and } j \textrm{ is even.} \\
\end{array}\right.$
\end{center}

Thus, $\phi$ is a bijection and in each entry of $(x_i)$ it behaves like an isomorphism or a half-isomorphism. Hence, $\phi$ is a half-automorphism of $G_\infty$. Since $f$ is a non-special half-isomorphism, there are $x,y \in G$ such that $f^{-1}(x\cdot y) \not \in \{f^{-1}(x)*f^{-1}(y),f^{-1}(y)*f^{-1}(x)\}$. Then,

\begin{center}
$\phi^{-1}((x)_{i=1}^\infty\bullet(y)_{i=1}^\infty) \not \in \{\phi^{-1}((x)_{i=1}^\infty)\bullet \phi^{-1}((y)_{i=1}^\infty),\phi^{-1}((y)_{i=1}^\infty)\bullet \phi^{-1}((x)_{i=1}^\infty)\}$.
\end{center}
 
Therefore, $\phi$ is a non-special half-automorphism of $G_\infty$.

In example~\ref{ex1}, we have loops $C_6 = (G,*)$ and $L=(G,\cdot)$ for the conditions above, hence the loop $G_\infty$ has a non-special half-automorphism. Note that $\mathit{Half}(G_\infty)$ is not a group.

In the following example, a non-special half-isomorphism between a right Bol loop and a group is provided. This example is obtained by using MACE4 \cite{mace}.

\begin{exem} 
\label{ex61}
Let $G = \{1,2,...,8\}$ and consider the following Cayley tables of $(G,*)$ and $(G,\cdot)$:

\begin{center}

\begin{minipage}{.4\textwidth}
\centering

\begin{tabular}{c|cccccccc}
$*$&1&2&3&4&5&6&7&8\\
\hline
1&1&2&3&4&5&6&7&8\\
2&2&1&4&6&3&5&8&7\\
3&3&4&1&2&7&8&5&6\\
4&4&3&2&8&1&7&6&5\\
5&5&6&7&1&8&2&3&4\\
6&6&5&8&7&2&1&4&3\\
7&7&8&5&3&6&4&1&2\\
8&8&7&6&5&4&3&2&1\\
\end{tabular}

\end{minipage} 
 \begin{minipage}{.4\textwidth}
\centering
\begin{tabular}{c|cccccccc}
$\cdot$&1&2&3&4&5&6&7&8\\
\hline
1&1&2&3&4&5&6&7&8\\
2&2&1&4&3&6&5&8&7\\
3&3&4&1&2&7&8&5&6\\
4&4&3&2&1&8&7&6&5\\
5&5&7&6&8&1&3&2&4\\
6&6&8&5&7&2&4&1&3\\
7&7&5&8&6&3&1&4&2\\
8&8&6&7&5&4&2&3&1\\
\end{tabular}

\end{minipage}

\end{center}

We have $(G,*) = L$ as a right Bol loop and $(G,\cdot)$ being isomorphic to $D_8$, which is the dihedral group of order $8$. The permutation $f = (3 \, 5 \, 7)(4\, 6\, 8)$ of $G$ is a half-isomorphism of $L$ into $D_8$. Since $|K(L)| = 56$ and $|K(D_8)| = 40$, $f$ is a non-special half-isomorphism by Theorem~\ref{teo31}. Since $L$ and $D_8$ are right Bol loops, $G_\infty$ is also a right Bol loop, and from the previous construction we have a non-special half-automorphism in a right Bol loop of infinite order.\qed
\end{exem}

\section*{Acknowledgments}
Some calculations in this work have been made by using the finite model builder MACE4, developed by McCune \cite{mace}, and the LOOPS package \cite{NV1} for GAP \cite{gap}.

\end{document}